\documentclass[12pt]{article}
\usepackage{amsmath, mathrsfs}
\usepackage{amsfonts}
\usepackage{amsthm}
\usepackage{graphicx}
\graphicspath{{Figures/}} 
\usepackage{overpic}
\usepackage{amssymb} 
\usepackage{pictex}
\usepackage{rotating}  
\usepackage{color}
\usepackage{cite} 

\usepackage{enumitem} 

\usepackage{accents} 

\usepackage{yhmath} 
\usepackage{etex} 
\definecolor{refkey}{rgb}{0,0,1}
\definecolor{labelkey}{rgb}{0,0,1}
\numberwithin{equation}{section}

\newtheorem{theorem}{Theorem}[section]

\newtheorem{lemma}[theorem]{Lemma}
\newtheorem{corollary}[theorem]{Corollary}
\newtheorem{Definition}[theorem]{Definition}
\newtheorem*{theorem*}{Theorem}
\newenvironment{definition}{\begin{Definition}\rm}{\end{Definition}}
\newtheorem{Remark}[theorem]{Remark}
\newenvironment{remark}{\begin{Remark}\rm}{\end{Remark}}
\newtheorem{RHproblem}[theorem]{RH problem}

\newtheorem{Example}[theorem]{Example}
\newenvironment{example}{\begin{Example}\rm}{\end{Example}}

\newcommand{\C}{\mathbb{C}}
\newcommand{\D}{\mathbb D}

\newcommand{\R}{\mathbb{R}}

\newcommand{\CC}{\mathcal C}

\newcommand{\LL}{\mathcal L}

\newcommand{\OO}{\mathcal O}

\newcommand{\VV}{\mathcal V}

\newcommand{\eps}{\epsilon}

\newcommand{\p}{\partial}
\newcommand{\vphi}{\varphi}

\usepackage{color}

\def \deg{\mbox{{\rm deg} }}
\def\det{\mathop{\mathrm{det}}\nolimits}

\def\supp{\mathop{\mathrm{supp}}\nolimits}

\renewcommand{\bar}{\overline}
\renewcommand{\tilde}{\widetilde}
\renewcommand{\hat}{\widehat}

\usepackage[bookmarksopen, naturalnames]{hyperref}

\begin{document}
\title{Weighted holomorphic polynomial approximation}
\author{S. Charpentier, N. Levenberg\thanks{Supported by Simons Foundation grant No. 707450}
 \ and F. Wielonsky}


\maketitle

\begin{abstract}

For $G$ an open set in $\C$ and $W$ a non-vanishing holomorphic function in $G$, in the late 1990's, Pritsker and Varga \cite{PV1} characterized pairs $(G,W)$ having the property that any $f$ holomorphic in $G$ can be locally uniformly approximated in $G$ by weighted holomorphic polynomials $\{W(z)^np_n(z)\}, \ deg(p_n)\leq n$. We further develop their theory in first proving a quantitative Bernstein-Walsh type theorem for certain pairs $(G,W)$. Then we consider the special case where $W(z)=1/(1+z)$ and $G$ is a loop of the lemniscate 
$\{z\in \C: |z(z+1)|=1/4\}$. We show the normalized measures associated to the zeros of the $n-th$ order Taylor polynomial about $0$ of the function $(1+z)^{-n}$ converge to the weighted equilibrium measure of $\bar G$ with weight $|W|$ as $n\to \infty$. This mimics the motivational case of Pritsker and Varga  \cite{PV3} where $G$ is the inside of the Szeg\H o curve and $W(z)=e^{-z}$. Lastly, we initiate a study of weighted holomorphic polynomial approximation in $\C^n, \ n>1$.
 
\end{abstract}



\bigskip

\section{Introduction} There is a well-developed theory of approximation by weighted polynomials $w(z)^np_n(z)$ in one complex variable where $w:K\to \R^+$ is a 
nonnegative weight function on a compact set $K\subset \C$ (or a closed set $K$ with additional hypotheses on $w$). This theory is highly developed for certain subsets of the real line where a major issue is in finding the support $S_w$ of the weighted equilibrium measure and its relation to uniform approximation of continuous functions, c.f., \cite{ST} Chapter VI. Pritsker and Varga \cite{PV1} considered the case of a ``holomorphic weight'' $W(z)$ in the sense that they studied approximation by weighted holomorphic polynomials $W(z)^np_n(z)$. Precisely, let $G$ be an open set in $\C$ and $W$ be a non-vanishing holomorphic function in $G$. 

\begin{definition} The pair $(G,W)$ is said to have the {\it approximation property} if any $f$ holomorphic in $G$ can be locally uniformly approximated in $G$ by weighted holomorphic polynomials; i.e., for any $K$ compact in $G$, there exist a sequence $\{P_n\}$ of polynomials $P_n \in \C[z]$ with $\deg P_n \leq n$ (we write $P_n \in \mathcal P_n$) such that 
$$\|f-W^{n}P_{n}\|_{K}:=\sup_{z\in K}|f(z)-W(z)^nP_n(z)|\to0.$$
\end{definition}

As pointed out in \cite{PV1}, for $W\equiv 1$, $G$ should be a finite or countable disjoint union of simply connected domains, and thus they only considered such $G$. Let $Q:=-\log |W|$. We state the main result of \cite{PV1}.

\begin{theorem}[{\cite[Theorem 1.1]{PV1}}] \label{pvmain}  
The pair $(G,W)$ has the approximation property if and only if there exists a probability measure $\mu$ supported on $\partial G$ and a real constant $F$ such that 
$$
 U^{\mu}(z)+Q(z)=F,\quad z\in G.
$$

\end{theorem}

\noindent Here, for a measure $\nu$, 
$$U^{\nu}(z):=\int \log \frac{1}{|z-\zeta|}d\nu(\zeta)$$ 
is the logarithmic potential of $\nu$.
If $W$ is continuous on $\bar G$, the measure $\mu$ in Theorem \ref{pvmain} is the solution $\mu^{\bar G,Q}$ of the weighted equilibrium problem on $\bar G$ with weight $Q$ (see section 2). Thus the approximation property implies that $\supp\mu^{\bar G,Q}\subset\p G$.

Many interesting classes of examples are given in their paper. For example, if $W(z)=z$, for a disk $G=D(a,r)=\{z\in \C: |z-a|<r\}$ centered at a point $a$ on the positive real axis, the pair $(G,W)$ has the approximation property if and only if $r<a/3$. 

In this paper, we further develop the theory of weighted holomorphic polynomial approximation. Our first main goal, in the next section, is to give a type of quantitative version of Theorem \ref{pvmain} in the spirit of Bernstein-Walsh (Theorem \ref{bwpv}). For $W$ entire, $K\subset \C$ compact, $R<0$, and $Q:=-\log |W|$, define
$$
K_{R}=\{z\in\C,~U^{\mu^{K,Q}}(z)+Q(z)-F_{K,Q}> R\}
$$
(see section \ref{Sec-one-var} for this (standard) notation in weighted potential theory). For $f$ continuous on $K$, define, for $n=1,2,...$, 
$$d_{n}^{W}(f,K):=\inf \{ \|f-W^{n}p_{n}\|_{K}: p_n \in \mathcal P_n\}.$$

\begin{theorem*}
Let $K\subset\C$ be a compact subset of a domain $G$ with $\C \setminus K$ connected such that $(G,W)$ is an approximation pair.
Let $f$ be a function continuous in $K$, and let $R<0$.
Then
\begin{equation}\label{rate}
\limsup_{n} d_{n}^{W}(f,K)^{1/n}\leq e^{R}.
\end{equation}
if and only if $f$ is the restriction to $K$ of a function holomorphic in the connected component of $K_{R}$ that contains $K$.
\end{theorem*}

The technique of proof of the ``if'' direction follows that of \cite{PV1}: we use polynomial interpolation of $f/W^n$ at $n-$th order weighted Fekete points 
of $K$ and the Hermite remainder formula. For the ``only if'' direction, the standard proof considering a telescoping series $p_0 +\sum_n (W^np_n -W^{n-1}p_{n-1})$ 
where $\{W^np_n\}$ are asymptotically optimal approximators on $K$ to $f$ requires some care as the difference $W^np_n -W^{n-1}p_{n-1}$ is not a weighted holomorphic polynomial. A special case of this result was stated as Theorem 3.4 in \cite{PV3} (see the next paragraph). There is a slight gap in their proof of the ``only if'' direction (in applying their Corollary 4.2) but 
their result remains valid using our proof. 

The article \cite{PV3} considered the case where $W(z)=e^{-z}$ and $G$ is the ``inside'' of the Szeg\H o curve 
$$S:=\{z\in \C: |ze^{1-z}|=1 \ \hbox{and} \ |z|\leq 1\}.$$
Their starting point was work of Szeg\H o \cite{S} on the asymptotic zero distribution of the normalized
partial sums $\{s_n(nz):=\sum_{k=0}^n (nz)^k/k!\}$ of the exponential function. They derived an asymptotic formula for the weighted normalized
partial sums $\{e^{-nz}s_n(nz)\}$ which led them to a special case of an approximation pair $(G,W)=(G,e^{-z})$. In section 3 we replace $e^{-z}$ by $1/(1+z)$ and we obtain analogues of the results in \cite{PV3} for $G$ being the inside of the right loop of the lemniscate 
$$\{z\in \C: |z(z+1)|=1/4\}.$$

There is a third paper of Pritsker of interest. In \cite{PV2}, he works on a compact set $E$ with $W\in A(E)$, the algebra of continuous functions on $E$ which 
are holomorphic on the interior $E^o$ of $E$. Defining $A(E,W)$ to be the uniform closure of sequences of weighted holomorphic polynomials $\{W^np_n\}$ on $E$, in an attempt to determine when a weighted Mergelyan property holds, namely, $A(E,W) =A(E)$ he introduces the intermediate algebra $[W(z),zW(z)]$ generated by $W(z)$ and $zW(z)$; i.e., the uniform closure of bivariate holomorphic polynomials in 
$W(z)$ and $zW(z)$ on $E$. Thus 
$$A(E,W) \subset [W(z),zW(z)] \subset A(E).$$
We make two important observations:
\begin{enumerate}
\item $[W(z),zW(z)]$ consists of all uniform limits on $E$ of {\it linear combinations } of weighted polynomials; and 
\item $[W(z),zW(z)]$ contains the constants while this is not necessarily the case for $A(E,W)$.
\end{enumerate}

\noindent With respect to this second item, we mention that the proof of the ``only if'' direction of Theorem \ref{pvmain} only uses the hypothesis that the constant function $1$ (or any single non-vanishing holomorphic function in $G$) can be locally uniformly approximated by weighted holomorphic polynomials $\{W^np_n\}$ on $G$. We should mention that Pritsker and Varga also studied the problem of local uniform {\it rational} approximation of functions holomorphic in $G$ by weighted rational functions; see \cite{PV4}. 

Using these univariate results as motivation, we begin a study of the multivariate situation in section 4. Tom Bloom and his students M. Branker and J. Callaghan \cite{Br}, \cite{Ca} proved multivariate analogues of weighted polynomial approximation of continuous functions on certain sets in $\R^n, \ n\geq 2$ for specific symmetric weights using pluripotential theory. At first glance, weighted holomorphic polynomial approximation in several complex variables appears unnatural. However, in considering multivariate generalizations of the algebras [$W(z),zW(z)]$ and $A(E)$, something can be said. 

Given a bounded, pseudoconvex domain $G\subset \C^n$, we say $G$ has the {\it Mergelyan property} if each function in the algebra $A(\bar G)$ of continuous functions on $\bar G$ which are holomorphic on $G$ can be uniformly approximated on $\bar G$ by functions that are holomorphic on a neighborhood of $\bar G$. In particular, this holds if $A(\bar G)$ coincides with the uniform algebra $P(\bar G)$; i.e., uniform limits of holomorphic polynomials on $\bar G$. Necessary and sufficient conditions so that $G$ has the Mergelyan property 
are not known. Following \cite{W}, we state a result giving sufficient conditions on a finite family of functions in $A(\bar G)$ to be a set of generators of $A(\bar G)$ for special $G$ (Theorem \ref{wermer}). As a corollary, for certain weights $W$ on $G\subset \C^n$ we will have an analogue of ``$[W(z),zW(z)]=A(\bar G)$'' but not necessarily ``$A(\bar G,W)=[W(z),zW(z)]$''. In section 5 we give some explicit examples of generators of $A(\bar G)$. 

\section{Some results in one variable}\label{Sec-one-var} Let $K\subset \C$ be closed and let $w$ be a nonnegative, upper semicontinuous function with
$\{z\in K:w(z)>0\}$ nonpolar. If $K$ is unbounded, we require $|z|w(z)\to 0 \ \hbox{as} \ |z|\to \infty, \ z\in K.$ Let $Q:=-\log w$. For a probability measure $\tau $ supported in $K$ (we write $\tau \in \mathcal M(K)$), the logarithmic energy is $I(\tau):=\int_K \int_K \log \frac {1}{|z-t|}d\tau(t)d\tau(z)$ and we define the weighted energy
	$$I^w(\tau):=\int_K \int_K \log \frac {1}{|z-t|w(z)w(t)}d\tau(t)d\tau(z)=I(\tau) +2\int_KQd\tau.$$ 
	Then there exists a unique $\mu^{K,Q}\in \mathcal M(K)$ with 
$$I^w(\mu^{K,Q})= \inf_{\tau \in \mathcal M(K)}I^w(\tau)=:V_w.$$ 
Let $S_w$ denote the support of $\mu^{K,Q}$. 
Setting
	$$F_{K,Q}:=I^w(\mu^{K,Q})-\int_K Q d\mu^{K,Q}=V_w-\int_K Q d\mu^{K,Q},$$
the measure $\mu^{K,Q}$ is characterized by the Frostman inequalities,
\begin{align}
U^{\mu^{K,Q}}(z)+Q(z) & \geq F_{K,Q},\qquad\text{q.e.}~ z\in K,
\\[5pt]
U^{\mu^{K,Q}}(z)+Q(z) & \leq F_{K,Q},\qquad z\in S_{w}
\end{align}
where q.e. means except for a set of capacity zero. Moreover, we have 
	$$
V_{K,Q}^{*}=-U^{\mu^{K,Q}}+F_{K,Q}
$$
where 
$$V_{K,Q}(z)=\sup\{\frac{1}{deg p}\log|p(z)|: p\in \cup_n \mathcal P_n, \ ||w^{deg p} p||_K \leq 1\}$$
and $V_{K,Q}^{*}(z):=\limsup_{\zeta \to z}V_{K,Q}(\zeta)$. We write $V_K$ if $Q\equiv 0$. 

We say $K$ is regular if $V_K=V_K^*$; i.e., $V_K$ is continuous. If $K$ is regular and $Q$ is continuous, then $V_{K,Q}$ is continuous. For simplicity, we 
will assume our compact sets are regular.	For a sequence of weighted monic Fekete polynomials $F_{n}$ with respect to $K,Q$, i.e., $F_n(z)=\prod_{j=1}^n(z-t_j^{(n)})$ where $t_1^{(n)},..., t_n^{(n)}\in K$ attain the supremum  in 
$$\sup_{z_1,...,z_n\in K}\prod_{1\leq i<j\leq n}|z_i-z_j|w(z_i)w(z_j),$$
it is known that
\begin{equation}\label{wtdfek}
\lim_{n}\|w^{n}F_{n}\|^{1/n}_{K}=e^{-F_{K,Q}} \ \hbox{and} \
\frac1n\log|F_{n}(z)|\to -U^{\mu^{K,Q}}(z),\qquad z\in\C\setminus K
\end{equation}
where this latter convergence is locally uniform (\cite{ST} p. 150). Furthermore, for $P_n\in \mathcal P_n$ we have a weighted Bernstein-Walsh inequality (\cite{ST} p. 153) :
\begin{equation}\label{ineq-BW}
|P_{n}(z)|\leq\|w^{n}P_{n}\|_{S_{w}}e^{n(-U^{\mu^{K,Q}}(z)+F_{K,Q})},\qquad z\in\C.
\end{equation}

Let $G$ be an open set in $\C$ and let $W$ be a non-vanishing holomorphic function in $G$. In this setting, $w:=|W|$ and $Q:=-\log w = -\log |W|$.
	\begin{lemma}\label{lem-G-K}
Let $(G,W)$ be a pair with the approximation property, and let $K$ be a regular compact subset of $G$. Then
\begin{equation}
U^{\mu^{K,Q}}(z)+Q(z)=F_{K,Q},\quad z\in K,
\end{equation}
and
\begin{equation}
U^{\mu^{K,Q}}(z)+Q(z)<F_{K,Q},\quad z\in \bar G\setminus K.
\end{equation}
\end{lemma}
\begin{proof}
For the first part, we have
$$Q(z)\geq V_{K,Q}(z),~z\in K,\quad\text{ and }\quad V_{K,Q}(z)\geq V_{\bar G,Q}(z)=Q(z),~z\in G,$$ 
hence, for $z\in K$,
$$
V_{K,Q}(z)=Q(z),\quad\text{or }\quad U^{\mu^{K,Q}}(z)+Q(z)=F_{K,Q}.
$$

For the second part, we consider the function $U^{\mu^{\bar G,Q}}-U^{\mu^{K,Q}}$ which is superharmonic outside of $K$ and vanishes at infinity. Note $U^{\mu^{K,Q}}$ restricted to $\p K$ is continuous and hence it is continuous on $\p K$ as a function in $G$. We have
$$
U^{\mu^{\bar G,Q}}(z)-U^{\mu^{K,Q}}(z)=[F_{\bar G,Q}-Q(z)]-[F_{K,Q}-Q(z)]=F_{\bar G,Q}-F_{K,Q},\quad
z\in\p K,
$$
hence, by the minimum principle for superharmonic functions applied on the complement of $K$,
$$
U^{\mu^{\bar G,Q}}(z)-U^{\mu^{K,Q}}(z)\geq F_{\bar G,Q}-F_{K,Q},\quad z\in\bar \C\setminus K.
$$
If equality would hold somewhere in $\bar \C\setminus K$, then the function $U^{\mu^{\bar G,Q}}-U^{\mu^{K,Q}}$ would be constant, hence equal to $0$. As this is not the case (e.g., $\mu^{\bar G,Q}$ is supported on $\partial G$), 
$$
U^{\mu^{\bar G,Q}}(z)-U^{\mu^{K,Q}}(z)> F_{\bar G,Q}-F_{K,Q},\quad z\in\bar \C\setminus K.
$$
On $\bar G\setminus K$, we get
$$
-Q(z)-U^{\mu^{K,Q}}(z)>-F_{K,Q},\qquad z\in\bar G\setminus K,
$$
or
$$
U^{\mu^{K,Q}}(z)+Q(z)<F_{K,Q},\qquad z\in\bar G\setminus K.
$$
\end{proof}
\begin{remark}
 From Lemma \ref{lem-G-K} follows that if $(G,W)$ is an approximation pair then, for any open subset $G'\subset G$, $(G',W)$ is also an approximation pair (since, by the above lemma, Theorem \ref{pvmain} applies).
\end{remark}

We will need the following result (cf., \cite{BL}, Lemma 7.3).
\begin{lemma}\label{conv-V}
Let $K$ be a compact set and $Q_{j}$ a sequence of weights on $K$ with $Q_{j}\uparrow Q$. Then
$$
\lim_{j}V_{K,Q_{j}}(z)=V_{K,Q}(z),\qquad z\in\C.
$$
If $K$ is regular and $Q_{j}$ and $Q$ are continuous on $K$, then the convergence is locally uniform.
\end{lemma}


As in the introduction, for $f$ continuous on $K$, define
$$d_{n}^{W}(f,K):=\inf \{ \|f-W^{n}p_{n}\|_{K}: p_n \in \mathcal P_n\},$$
and for $R<0$, we define 
$$
K_{R}=\{z\in\C,~U^{\mu^{K,Q}}(z)+Q(z)-F_{K,Q}> R\},
$$
which always contains $K$.
The boundary of $K_{R}$ is denoted by
\begin{equation}\label{def-ER}
E_{R}=\{z\in\C,~U^{\mu^{K,Q}}(z)+Q(z)-F_{K,Q}=R\}.
\end{equation}

In the next theorem, we assume that the weight $W$ is an entire function. This implies, in particular, that the function $zW(z)$ is unbounded in $\C$, and thus $U^{\mu^{K,Q}}(z)+Q(z)$ is unbounded below (note $U^{\mu^{K,Q}}(z)$ behaves like $-\log|z|$ near infinity), so that, for any $R<0$, the level set $E_{R}$ is not empty.
Also, without loss of generality, we normalize $W$ so that $|W|\leq1$ on $K$. 
\begin{theorem} \label{bwpv} 
Let $K\subset\C$ be a compact subset of a domain $G$ with $\C \setminus K$ connected such that $(G,W)$ is an approximation pair.
Let $f$ be a function continuous in $K$, and let $R<0$.
Then
\begin{equation}\label{rate}
\limsup_{n} d_{n}^{W}(f,K)^{1/n}\leq e^{R}.
\end{equation}
if and only if $f$ is the restriction to $K$ of a function holomorphic in the connected component of $K_{R}$ that contains $K$.
\end{theorem}
\begin{proof}
Denote by $K_{R}^{(K)}$ the component of $K_{R}$ that contains $K$, and by $E_{R}^{(K)}$ its boundary.
Assume $f$ is holomorphic in $K_{R}^{(K)}$. We consider the interpolants $P_{n}(z)$ of $f(z)/W(z)^{n}$ at the weighted Fekete points on $K$. 
For $F_{n}$ the $n-$th order monic Fekete polynomial associated to $K,Q$, for $z\in K$, and some $R'$ with $R<R'<0$, we have, by Hermite's remainder formula,
$$
f(z)-W(z)^{n}P_{n}(z)=\frac{W(z)^{n}}{2i\pi}\int_{E_{R'}^{(K)}}\frac{F_{n}(z)}{F_{n}(\xi)}\frac{f(\xi)}{W(\xi)^{n}}\frac{d\xi}{\xi-z}.
$$
From the asymptotics of $W^{n}F_{n}$ in (\ref{wtdfek}), for $z\in K$ and $\xi \in E_{R'}^{(K)}$, we have
$$\frac{|W(z)^{n}F_{n}(z)|}{|W(\xi)^{n}F_{n}(\xi)|} \asymp \exp [n\bigl(-F_{K,Q}-(-F_{K,Q}-R')\bigr)]=e^{nR'}.$$
Since $|f|$ is bounded on $K$ and $|\xi-z|$ is bounded below for $z\in K$ and $\xi \in E_{R'}^{(K)}$, 
$$\limsup_{n} ||f-W^nP_n||_K^{1/n} \leq e^{R'}.$$
Letting $R'$ decrease to $R$ gives (\ref{rate}).

Conversely, assume (\ref{rate}) holds true, that is, there exists a sequence of polynomials $p_{n}$ such that
$$
\limsup_{n} \|f-W^{n}p_{n}\|_{K}^{1/n}\leq e^{R}.
$$
We choose any $R<R_{1}<0$ and show that $f$ is holomorphic in $K_{R_{1}}^{(K)}$, which is sufficient.
We adapt the usual proof of the indirect implication in the Bernstein-Walsh theorem. Thus we want to show that the series 
$$p_{0}+\sum_{n}(W^{n}p_{n}-W^{n-1}p_{n-1})$$ 
converges uniformly in $K_{R_{1}}$. For given $\eps>0$ and $\rho<0$, let $q_{\eps n}$ be a sequence of polynomials of degree $O(\eps n)$ such that
\begin{equation}\label{BW-W}
\limsup_{n}\|W-q_{\eps n}\|_{K_{R_{1}}}^{1/n}\leq e^{\rho}.
\end{equation}
The usual Bernstein-Walsh theorem shows that such a sequence exists (recall that $W$ is assumed to be an entire function).
We have
$$
\|W^{n}p_{n}-W^{n-1}p_{n-1}\|_{K_{R_{1}}}\leq
\|W^{n-1}(q_{\eps n}p_{n}-p_{n-1})\|_{K_{R_{1}}}
+\|(W-q_{\eps n})W^{n-1}p_{n}\|_{K_{R_{1}}}.
$$
Without loss of generality, we may assume that $R_{1}$ is such that $W$ does not vanish on $E_{R_{1}}$.
Thus, we may multiply the upper bound by $W$ to get
$$
\|W^{n}(q_{\eps n}p_{n}-p_{n-1})\|_{K_{R_{1}}}
+\|(W-q_{\eps n})W^{n}p_{n}\|_{K_{R_{1}}}.
$$
Because of (\ref{ineq-BW}) and (\ref{BW-W}), the second norm decreases geometrically to zero as soon as $\rho<R_{1}$.
We next consider the first norm. Since $|W|\leq1$ on $K$, we have $Q\geq0$ on $K$. We define the weights 
$$
Q_{\eps}=(1+\eps)^{-1}Q\uparrow Q\quad\text{as}\quad\eps\downarrow0.
$$
By Lemma \ref{conv-V}, for any $R_{2}<R_{1}$, there exists $\eps>0$ sufficiently small so that
\begin{equation}\label{ineq-R2}
U^{\mu^{Q_{\eps},K}}(z)+Q_{\eps}(z)-F_{K,Q_{\eps}}\geq R_{2},\quad z\in K_{R_{1}}.
\end{equation}
Since $q_{\eps n}p_{n}-p_{n-1}$ is of degree $n+O(\eps n)\sim(1+\eps)n$ and
$$
W^{n}(q_{\eps n}p_{n}-p_{n-1})=(W^{1/(1+\eps)})^{(1+\eps)n}(q_{\eps n}p_{n}-p_{n-1}),
$$
we get by using the weighted Bernstein-Walsh inequality (\ref{ineq-BW}) with weight $W^{1/(1+\eps)}$, along with (\ref{ineq-R2}), that
\begin{align*}
\|W^{n}(q_{\eps n}p_{n}-p_{n-1})\|_{K_{R_{1}}} & \leq e^{-(1+\eps)nR_{2}}\|W^{n}(q_{\eps n}p_{n}-p_{n-1})\|_{K}
\\[5pt]
& \leq e^{-(1+\eps)nR_{2}}\left(\|W^{n+1}p_{n}-W^{n}p_{n-1}\|_{K}+ \|W^{n}(q_{n}-W)p_{n}\|_{K}\right)
\\[5pt]
& \leq 2\|W\|_{K}e^{n(-(1+\eps)R_{2}+R)}+2\|f\|_{K}e^{n(-(1+\eps)R_{2}+\rho)}.
\end{align*}
Hence, we see that three inequalities need to be satisfied, namely,
$$
R<(1+\eps)R_{2},\qquad\rho<(1+\eps)R_{2},\qquad\rho<R_{1}.
$$
The first one can be satisfied by taking $R_{2}$ sufficiently close to $R_{1}$ and $\eps$ small. The second one can also be satisfied because $\rho$ can be chosen as any large negative number (since $W$ is an entire function). Finally, the second inequality implies the third one since $(1+\eps)R_{2}<R_{1}$.
\end{proof}
We illustrate the previous theorem by describing the level lines $E_{R}$, $R<0$, in the classical case of incomplete polynomials, see \cite{G, L, SV} and \cite[Section VI.1]{ST}. Let $W(z)=z$, i.e.\ $Q(z)=-\log|z|$. It is known that weighted approximation, with weight $W$ of any continuous function $f$ on a closed interval $[c,1]$ is possible if and only if $c>1/4$ (if $c=1/4$, the additional assumption that $f(1/4)=0$ is needed). Hence, we consider the case 
$K=[c,1]$, $1/4<c<1$. The weighted equilibrium measure is known, namely
$$
\mu^{K,Q}=2\mu^{K}-\hat\delta_{0},
$$
where $\mu^{K}$ is the unweighted equilibrium measure of $K$, and $\hat\delta_{0}$ is the balayage of $\delta_{0}$ on $K$. Up to an additive constant, the difference $U^{\delta_{0}}(z)-U^{\hat\delta_{0}}(z)$, $z\in\C\setminus K$, is equal to the Green function $g_{\C\setminus K}(z,0)$ of the complement of $K$ with pole at 0, see \cite[Section II.5]{ST}. Hence, we are interested in the level lines of
$$
2U^{\mu^{K}}(z)-U^{\delta_{0}}(z)+g_{\C\setminus K}(z,0)+Q(z)=2U^{\mu^{K}}(z)+
g_{\C\setminus K}(z,0),
$$
or, equivalently, the level lines of
$$
g_{\C\setminus K}(z,0)-2g_{\C\setminus K}(z,\infty)=
\log\left|\frac{1-\bar{\vphi(0)}\vphi(z)}{\vphi(z)-\vphi(0)}\right|
-2\log\left|\vphi(z)\right|,
$$
where $\vphi$ is the conformal map from $\C\setminus[c,1]$ to $\C\setminus\D$,
$$
\vphi(z)=\frac{2}{1-c}\left[z-(1+c)/2+\sqrt{(z-c)(z-1)}\right].
$$
\begin{figure}[htb]
\centering
\includegraphics[height=6cm]{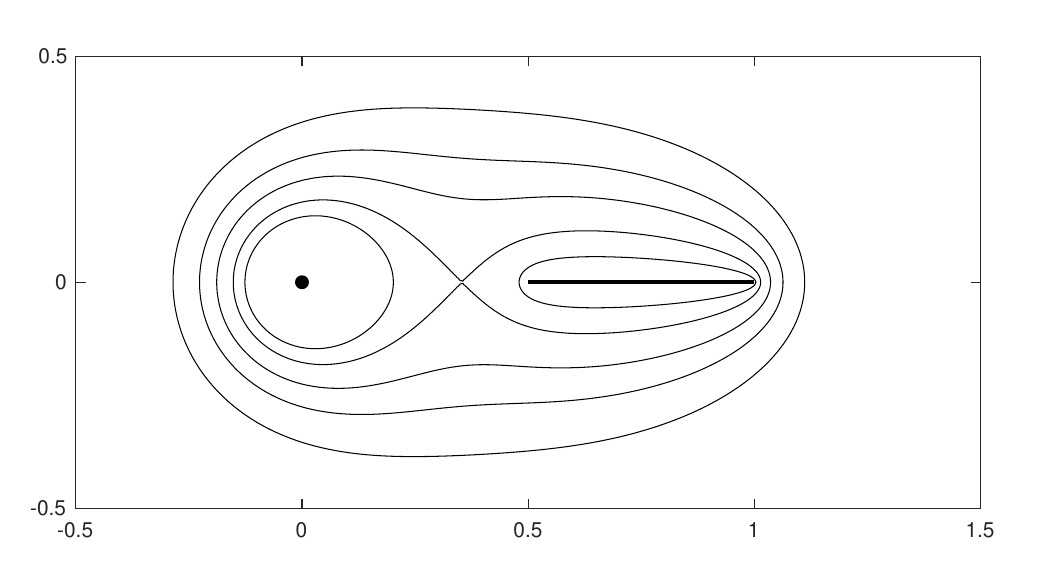}
\caption{A few level lines $E_{R}$, defined in (\ref{def-ER}), for the case of the segment $[1/2,1]$ and the weight $W(z)=z$.}
\label{Level-segm}
\end{figure}
A few of these level lines $E_{R}$ are shown in Figure \ref{Level-segm}. For $R<0$ close to 0, $E_{R}$ is made of two disjoint curves. For a particular value $R_{0}$ of $R$, the two curves merge at some real point 
and then, for $R<R_{0}$, the curve $E_{R}$ has a single component. When $E_R$ has a single component, $(K_R,W)$ is an approximation pair because, by the definition of $E_R$, $U^{\mu^{K,Q}}(z)+Q(z)$ is constant there so that the balayage of $\mu^{K,Q}$ onto $E_R$ satisfies the condition of Theorem \ref{pvmain}. We also remark that,  when $R<0$ decreases, the increase of the connected component of $K_{R}$ that contains $K$ can have a discontinuity, namely when two components of $K_{R}$ merge. This phenomenon does not occur with the classical (unweighted) Bernstein-Walsh theorem.

\begin{remark}  Let $z^{n}P_{n}$ be bounded on $[c,1]$, $1/4<c<1$. Then $z^{n}P_{n}\to0$ on $[0,c)$. 

To see this, by the Bernstein-Walsh inequality (\ref{ineq-BW}) for $K=[c,1]$ and $w(z)=|z|$, 
$$
|z^{n}P_{n}(z)|\leq\|z^{n}P_{n}\|_{[c,1]}e^{n(-U^{K,Q}(z)-Q(z)+F_{K,Q})}.$$
Now $U^{K,Q}(z)+Q(z)=F_{K,Q}$ on $[c,1]$ while 
$$U^{K,Q}(z)+Q(z)-F_{K,Q}>0,\quad z\in[0,c).
$$
Indeed, the support $S_w$ of the measure $\mu^{K,Q}$ is on $[c,1]$  (recall the Frostman inequalities (2.1) and (2.2)) while $U^{K,Q}(z)+Q(z)$ is strictly convex on $(0,c)$ and becomes large when $z\to0$.
\end{remark}

In a similar fashion, we observe the following.

\begin{remark} Let $D(a,r)=\{z\in \C: |z-a|<r\}$ be a Pritsker disk for $W(z)=z$, i.e., $a>0$ and $r<a/3$. Let $z^{n}P_{n}(z)\to f(z)$ uniformly in $D(a,r+c)$ for $c>0$ where $f$ is holomorphic. Then $f\equiv 0$.

To see this, on the boundary of $K:=\bar D(a,r+c/2)$ there must exist a point $b$ such that, for the weighted equilibrium problem with $Q=-\log |z|$ on 
$K$, one has $$U^{\mu^{K,Q}}(b)+Q(b)>F_{K,Q}$$ and, by continuity ($K$ is regular and $Q$ is continuous), the strict inequality is still true in a neighborhood $\VV_{b}$ of $b$, included in $D(a,r+c)$. By the Bernstein-Walsh inequality (\ref{ineq-BW}),
$$
|z^{n}P_{n}(z)|\leq\|z^{n}P_{n}\|_{\bar D(a,r+c/2)}e^{n(-U^{\mu^{K,Q}}(z)-Q(z)+F_{K,Q})}.
$$
The norms $\{\|z^{n}P_{n}\|_{\bar D(a,r+c/2)}\}$ are bounded, by convergence of $z^{n}P_{n}$ to $f$, and the exponential goes to $0$ for $z\in \VV_{b}$.
\end{remark}

\noindent This last remark shows that $(D(a,r),z)$ is an example of a {\it maximal} approximation pair, a notion we discuss in the next section.

\section{An analogue of the Szeg\H o curve and a related approximation pair}
In this section, we state some properties for an analogue of the classical Szeg{\H o} curve \cite{S}, where we replace the exponential function $e^{-z}$ with the rational function $1/(1+z)$. As we will see in Theorem \ref{scurve}, this special curve, denoted by $\LL_{+}$, consists of the right loop of the lemniscate of degree 2,
$$\LL=\LL_{+}\cup\LL_{-}=\{z\in\C,~|z(z+1)|=1/4\},$$ 
where $\LL_{+}$ (resp. $\LL_{-}$) denotes the loop around 0 (resp.\ around $-1$), see Fig.\ \ref{Lemn}. Let $G$ be the bounded domain with boundary $\LL_{+}$. 
\begin{figure}[htb]
\centering
\includegraphics[height=6cm]{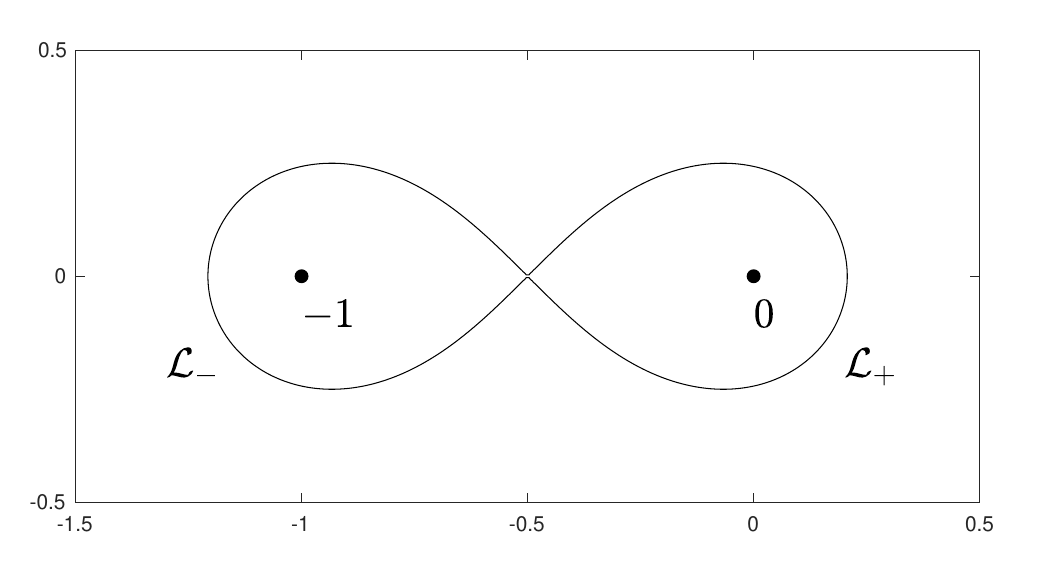}
\caption{The lemniscate $4|z(z+1)|=1$.}
\label{Lemn}
\end{figure}
We prove that the pair $(G,W)$ with $W(z)=1+z$ 
form a {\it maximal} approximation pair, see 4) of Theorem \ref{scurve}. For that, we will consider the weighted equilibrium measure $\mu^{\bar G,Q}$ with $Q(z)=-\log|1+z|$. Note that, for $z\in\LL_{+}$,
\begin{equation}\label{U+Q=4}
\log\frac{1}{|z|}+Q(z)=\log4\quad\iff\quad U^{\hat\delta_{0}}(z)+Q(z)=\log4,
\end{equation}
where $\hat\delta_{0}$ is the balayage of $\delta_{0}$ onto $\LL_{+}$. Hence, by the minimum principle for superharmonic functions, $U^{\hat\delta_{0}}(z)+Q(z)\geq\log4$, $z\in G$, which implies that $\mu^{\bar G,Q}=\hat\delta_{0}$ and $F_{\bar G,Q}=\log4$. Moreover, (\ref{U+Q=4}) implies that $U^{\hat\delta_{0}}$ is continuous as a function restricted to the support of $\hat\delta_{0}$, hence is continuous everywhere. Thus, by applying the maximum principle to the function $U^{\hat\delta_{0}}+Q$, harmonic in $G$, continuous in $\bar G$, we get that
\begin{equation}\label{eq-G-bar}
U^{\hat\delta_{0}}(z)+Q(z)=\log4, \qquad z\in \bar G.
\end{equation}
The following lemma is elementary but will be useful in the proof of our theorem. Since we have not found a reference in the literature, we give a proof.
\begin{lemma}\label{Lem-Tayl}
Let $f$ be a function holomorphic in the unit disk $D:=\{z:|z|<1\}$. Then Taylor's formula with integral remainder holds true in $D$, namely, for any $z\in D$, 
\begin{equation}\label{Taylor-rem}
f(z)=\sum_{k=0}^{n}\frac{f^{(k)}(0)}{k!}z^{k}+\frac{1}{n!}\int_\gamma f^{(n+1)}(t)(z-t)^{n}dt,
\end{equation}
where $\gamma$ is any path from 0 to $z$ in $D$.
\end{lemma}
\begin{proof}
First we note that (\ref{Taylor-rem}) holds true with $\gamma=[0,z]$. Indeed, it suffices to apply the usual Taylor's formula to the function of a real variable $s\in[0,1]\mapsto f(sz)$ and to make the change of variable $t=sz$, see also \cite[\S 7.1 p.125]{WW}. Now, the segment $[0,z]$ can be deformed to any contour from $0$ to $z$ in $D$ and the formula still holds since  each function $t\mapsto f^{(n+1)}(t)(z-t)^{n}$ is holomorphic. \end{proof}

Finally, for a compact set $K$, and a function $f$ defined on $K$, we define the essential sup of $f$ on $K$,
$$
\|f\|_{K}^{*}=\inf\{L:~|f(z)|\leq L~\text{ q.e.\ on }K\}. 
$$

Let $s_{n}((1+z)^{-n})$ denote the $n-th$ order Taylor polynomial about $0$ of the function $(1+z)^{-n}$.
\begin{theorem}\label{scurve}
The following hold: \\
1) As $n\to \infty$,
\begin{equation}\label{estim-diff}
1-(1+z)^{n}s_{n}((1+z)^{-n})=\frac{(-1)^{n+1}z}{\sqrt{n\pi}(2z+1)}(4z(1+z))^{n}
\left(1+\OO\left(\frac1n\right)\right),\quad z\in\bar G\setminus\{-1/2\}.
\end{equation}
2) Let $\tilde s_{n}((1+z)^{-n})$ be the constant multiple of $s_{n}((1+z)^{-n})$ normalized so that its leading coefficient is 1.
Then, the sequence of monic polynomials $\tilde s_{n}((1+z)^{-n})$ is asymptotically extremal for $\bar G$ (cf., \cite{MS}) and the weight $w(z)=|W(z)|=|1+z|$, that is
\begin{equation}\label{estim-extr}
\lim_{n\to\infty}\left(\|(1+z)^{n}\tilde s_{n}((1+z)^{-n}\right)\|_{\bar G}^{*})^{1/n}=\exp(-F_{\bar G,Q})=\frac14.
\end{equation}
3) 
Let $\nu_{n}=\frac{1}{n}\sum_{j=1}^n \delta_{z_{nj}}$ be the normalized measure associated to the zeros $\{z_{nj}\}$ of the polynomial $s_{n}((1+z)^{-n})$. Then
$$
\nu_{n}\to\mu^{\bar G,Q},\quad\text{ as } n\to\infty.
$$
4) The pair $(G,1+z)$ is an approximation pair which is maximal, in the sense that for any domain
$H$ containing $\bar G$, $(H,1+z)$ is not an approximation pair.
\end{theorem}
\begin{figure}[htb]
\centering
\includegraphics[height=6cm]{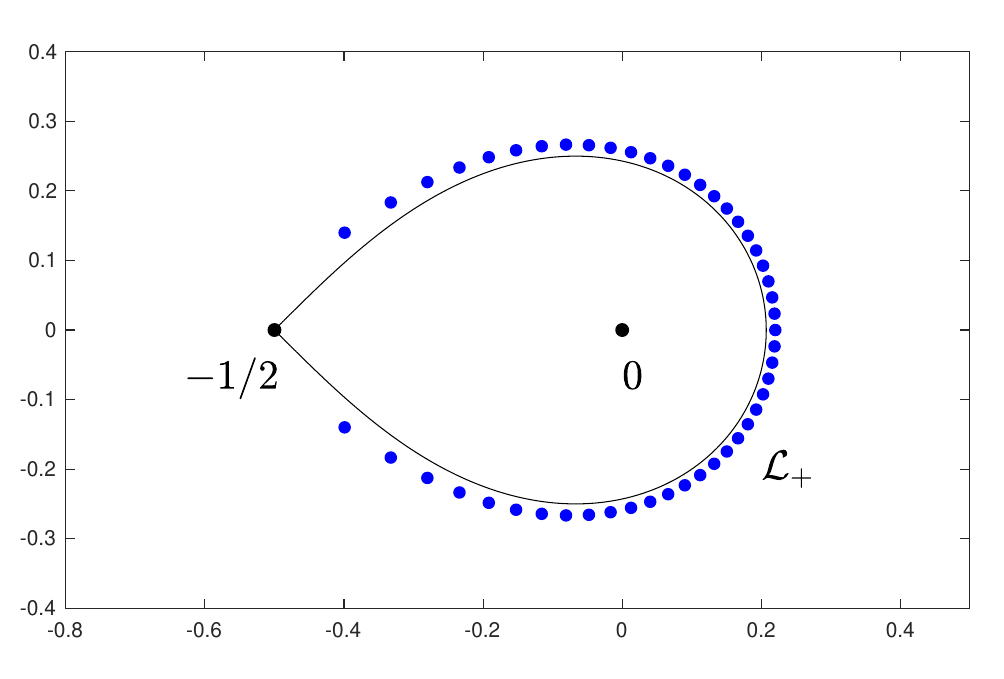}
\caption{The 50 zeros of $s_{50}((1+z)^{-50})$, and the right loop $\LL_{+}$ of the lemniscate $\LL$, an analogue of the Szeg\"o curve for the function $1/(1+z)$.}
\label{Zeros}
\end{figure}
\begin{proof}
1) 
 : By making use of Lemma \ref{Lem-Tayl}, one has, for $z\in \bar G\setminus\{-1/2\}$, 
\begin{align}
(1+z)^{-n}-s_{n}((1+z)^{-n}) & =(-1)^{n+1}n
\frac{(2n)!}{(n!)^{2}}\int_{\gamma}\frac{(z-t)^{n}}{(1+t)^{2n+1}}dt \label{Taylor}
\\[5pt]
& = (-1)^{n+1}\frac{\sqrt{n}}{\sqrt{\pi}}4^{n}\int_{\gamma}\frac{(z-t)^{n}}{(1+t)^{2n+1}}dt
\left(1+\OO\left(\frac1n\right)\right), \notag
\end{align}
where $\gamma$ is any path from 0 to $z$ avoiding $-1$. To get an estimate for the integral, we use the steepest descent method, see \cite{D}, \cite[Chapter 4]{O}. The integrand can be written in the form
${h(t)g(t)^{n}}
$, where
$$
h(t)=\frac{1}{1+t},\qquad g(t)=\frac{(z-t)}{(1+t)^{2}}.
$$
\begin{figure}[htb]
\centering
\includegraphics[height=9cm]{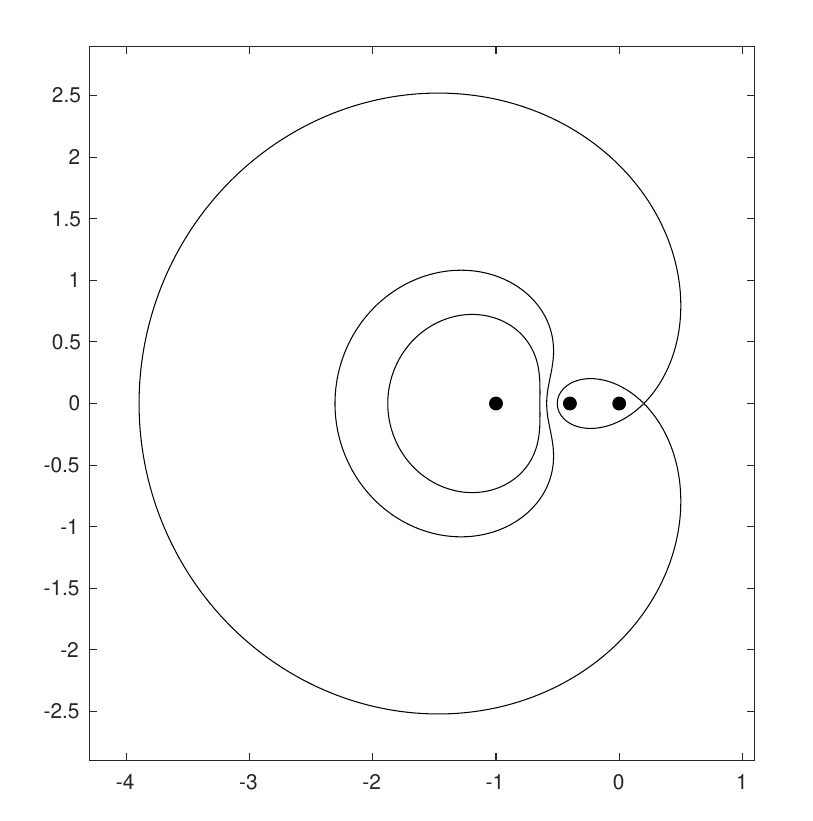}
\caption{Three level lines of $g(t)$ when $z=-0.4$, and the points $-1,z,0$. The level line 
$\CC_{1}\cup\CC_{2}$ with the inner loop $\CC_{1}$ is the one passing through the critical point $t_{0}=2z+1=0.2$.}
\label{Lev-lines}
\end{figure}
The function $g$ has a unique critical point at $t_{0}=2z+1$ with critical value $-1/(4(z+1))$. 
Figure \ref{Lev-lines} depicts a few level lines of $g(t)$ when $z=-0.4$. Independently of the choice of $z$, the critical level line passing through $t_{0}$ is made of two loops $\CC_{1}\cup\CC_{2}$ intersecting at $t_{0}$, the loop $\CC_{1}$ lying in the interior of $\CC_{2}$. We denote by $A_{0}$ the domain enclosed by the inner loop $\CC_{1}$, $A_{1}$ the domain exterior to the large loop $\CC_{2}$, and $A_{2}$ the remaining domain lying outside $\CC_{1}$ and inside $\CC_{2}$. 
Note that $-1\in A_{2}$, and, by Rouch\'e's theorem, $z\in A_{0}$. Moreover, the point 0 belongs to $A_{0}\cup A_{1}$ if and only if $4|z(z+1)|<1$, that is $z$ lies inside 
the lemniscate $\LL$ (depicted in Figure \ref{Lemn}). Note that
$0\in A_{0}$ if $z\in\LL_{+}$ and $0\in A_{1}$ if $z\in\LL_{-}$. Note also that when $z=-1/2$ then $t_{0}=0$ is the critical point of $g$.

If $z$ lies inside of $\LL_{+}$ or on $\LL_{+}\setminus\{-1/2\}$, then $0\in A_{1}\setminus\{t_{0}\}$, and it is clear that there is a path from 0 to $z$ such that the maximum of $|g(t)|$ is attained at 0 and only at this point, with $g'(0)\neq0$. Thus, the steepest descent method gives the estimate 
$$
\int_{0}^{z}\frac{(z-t)^{n}}{(1+t)^{2n+1}}dt=-
\frac{h(0)g(0)^{n+1}}{ng'(0)}\left(1+\OO\left(\frac1n\right)\right)
=\frac{z^{n+1}}{n(2z+1)}\left(1+\OO\left(\frac1n\right)\right)
$$
which implies (\ref{estim-diff}):
$$
1-(1+z)^{n}s_{n}((1+z)^{-n})=\frac{(-1)^{n+1}z}{\sqrt{n\pi}(2z+1)}(4z(1+z))^{n}
\left(1+\OO\left(\frac1n\right)\right),\quad  z\in\bar G\setminus\{-1/2\}.
$$
This will not be needed in the sequel, but we may notice that this estimate is still valid when $z$ lies outside of $\LL_{-}\cup\LL_{+}$ (then $0$ lies in $A_{2}$). If $z$ lies inside $\LL_{-}$ then $0\in A_{1}$ and the main contribution to the integral comes from the saddle point $t_{0}$.
\\
2) We first note that the leading coefficient of $s_{n}((1+z)^{-n})$ is
$$
\frac12\frac{(2n)!}{(n!)^{2}};
$$
Stirling's formula yields that its $n$-th root limit equals $4$. 
To prove (\ref{estim-extr}), we use (\ref{estim-diff}) where we notice that the one point set $\{-1/2\}$ has capacity 0. We get, for $z\in\LL_{+}\setminus\{-1/2\}$ and some $\theta\in[0,2\pi)$,
\begin{align}
(1+z)^{n}s_{n}((1+z)^{-n})\sim1+\frac{(-1)^{n+1}ze^{i\theta}}{\sqrt{n\pi}(2z+1)},
\end{align}
whose $n$-th root modulus tends to 1 as $n\to \infty$.\\
3) The estimate (\ref{estim-diff}) shows that the support of any limit measure of the sequence $\{\nu_{n}\}$ does not intersect the domain $G$. Hence, it suffices to apply \cite[Theorem 2.3(b)]{MS}.
\\
4) The fact that $(G,1+z)$ is an approximation pair is just a consequence of Theorem \ref{pvmain} and equation (\ref{eq-G-bar}). To prove that it is a maximal pair, first note that for $z$ outside of $\LL_{+}$,
$$
U^{\hat\delta_{0}}(z)+Q(z)=U^{\delta_{0}}(z)+Q(z) =-\log|z(z+1)|,
$$
and, thus, for $z$ inside $\LL_{-}$, it holds that $U^{\hat\delta_{0}}(z)+Q(z)>\log4$.
Together with the Bernstein-Walsh inequality (\ref{ineq-BW}), for $P_n \in {\mathcal P}_n$,
$$
|(1+z)^{n}P_{n}(z)|\leq\|(1+z)^{n}P_{n}\|_{\bar G}\exp(n(-U^{\hat\delta_{0}}(z)-Q(z)+\log4)),
\qquad z\in\C,
$$
we see that if a sequence $\{(1+z)^nP_{n}(z)\}$ is uniformly bounded in $\bar G$, then it has to decrease to 0 in $\LL_{-}$. Hence the only function holomorphic in a neighborhood of $\bar G$  which can be approximated uniformly on $\bar G$ by a weighted sequence $\{(1+z)^{n}P_{n}(z)\}$ is the zero function.
\end{proof}
\begin{remark} More generally, one can apply similar reasoning in replacing $1+z$ by the function $(1+z)^{\alpha}$ for $\alpha >0$. \end{remark}

Similarly to the case of the classical Szeg\H o curve, see \cite[Theorem 3.2]{PV3}, the property that $(G,1+z)$ is an approximation pair can be extended in the following way.
\begin{theorem}
1) Let $f$ be a function holomorphic in $G$, continuous in $\bar G\setminus\{-1/2\}$. Then, for any compact subset $K$ of $\bar G\setminus\{-1/2\}$, there exists a sequence of polynomials $\{P_{n}\}$, $\deg P_{n}\leq n$, such that
\begin{equation}\label{approx-K}
\|f(z)-(1+z)^{n}P_{n}(z)\|_{K}\to0\quad\text{as}\quad n\to\infty.
\end{equation}
2) Let $f$ be a function holomorphic in $G$, continuous in $\bar G$ with $f(-1/2)=0$. Then 
there exists a sequence of polynomials $\{P_{n}\}$, $\deg P_{n}\leq n$, such that
\begin{equation}\label{approx-K-2}
\|f(z)-(1+z)^{n}P_{n}(z)\|_{\bar G}\to0\quad\text{as}\quad n\to\infty.
\end{equation}
\end{theorem}
\begin{proof}
1) We may assume $K$ has connected complement. It follows from (\ref{estim-diff}) that the constant function 1 can be approximated as in (\ref{approx-K}). Multiplying (\ref{estim-diff}) by $W(z)=1+z$ shows that $W(z)=1+z$, and hence $z=1-W(z)$, can also be approximated in this fashion. Moreover if the functions $f,g$ are approximable by $W^{n}P_{n}, \ W^{n}Q_{n}$, then the product $fg$ is also approximable since
$$
fg-W^{2n}P_{n}Q_{n}=g(f-W^{n}P_{n})+f(g-W^{n}Q_{n})-(f-W^{n}P_{n})(g-W^{n}Q_{n}).
$$ 
Hence, all monomials $1,z,z^{2},...$ are approximable on $K$, and then, by Mergelyan's theorem, all functions $f$ satisfying the assumptions in 1).
\\
2) Multiplying (\ref{estim-diff}) by $(1+z)(2z+1)$, we see that $(1+z)(z+1/2)$ is approximable by weighted polynomials in $\bar G$. Hence the same is true for $(1+z)(z+1/2)(1+z)^{k}$, $k\geq0$. Thus, all polynomials $(1+z)(z+1/2)P_{n}(z)$ are approximable in $\bar G$. Now, in addition to the assumptions satisfied by $f$, we moreover assume that it is holomorphic at $-1/2$. Then the function $g(z)=f(z)/((1+z)(z+1/2))$ is holomorphic in $G$ and continuous in $\bar G$. Hence, by Mergelyan's theorem, it can be approximated by a sequence of polynomials $Q_{n}$ in $\bar G$, which implies that $f$ is approximable by the sequence $(1+z)(z+1/2)Q_{n}(z)$ in $\bar G$. Together with the previous assertion, it proves 2) when $f$ is holomorphic at $-1/2$. Finally, if $f$ is only supposed to be continuous at $-1/2$, again by Mergelyan's theorem, it can be approximated by a sequence of polynomials $Q_{n}$. Since $Q_{n}(-1/2)$ tends to $f(-1/2)=0$, the polynomials $Q_{n}$ can be replaced with $Q_{n}-Q_{n}(-1/2)$, 
which satisfy the assumptions of 2) and are holomorphic at $-1/2$. Hence, the first part of the proof can be applied.
\end{proof}

In analogy with the conjecture below Theorem 3.2 in \cite{PV3}, we conjecture that for $f$ holomorphic in $G$, continuous in $\bar G$ with $f(-1/2)\neq0$, weighted approximation is not possible.
\section{Several complex variables} In the general setting of a uniform algebra $\mathcal U$ on a compact Hausdorff space $X$, given a collection of elements 
$f_1,...,f_k\in \mathcal U$, let $[f_1,...,f_k|X]$ denote the smallest closed subalgebra of $\mathcal U$ containing the constants and $f_1,...,f_k$. If $[f_1,...,f_k|X]=\mathcal U$, we say that 
$f_1,...,f_k$ are a set of generators for $\mathcal U$. Thus, in the previous univariate setting, given $(G,W)$ with $W\in A(\bar G)$ and $W\not = 0$ on $\bar G$, with this notation we have 
$$[W,zW]=[W,zW|\bar G].$$
For $G\subset \C$ such that $\bar G$ has connected complement, $A(\bar G)=P(\bar G)$ by Mergelyan's theorem. Thus the question of whether $[W,zW|\bar G]=A(\bar G)$ reduces to the question of whether $W,zW$ are generators of $P(\bar G)$, and sufficient conditions are given in \cite[Proposition 2.5]{PV2}. We use this as motivation for a multivariate analogue. 

Given $K\subset \C^n$ compact, 
$$\hat K:=\{z \in \C^n: |p(z)|\leq ||p||_K \ \hbox{for all holomorphic polynomials} \ p\}$$
is the polynomial hull of $K$. If $\hat K =K$ we say $K$ is polynomially convex. The algebra $P(K)$ consists of uniform limits of holomorphic polynomials restricted to $K$. We begin by stating a generalization of Theorem 19.1 of \cite{W} (c.f., also \cite{SW}). The proof is essentially the same; for the reader's convenience, we include a sketch. 

\begin{theorem} \label{wermer} Let $G\subset \C^n$ be a bounded, pseudoconvex domain with the property that $\bar G$ is polynomially convex and 
$A(\bar G)=P(\bar G)$ (in particular, $G$ has the Mergelyan property). Let $f_1,...,f_k$ ($k \geq n$) be holomorphic in a neighborhood of $\bar G$ satisfy the following three properties:
\begin{enumerate}
\item $f_1,...,f_k$ separate points in $\bar G$;
\item the matrix $\bigl[ \frac{\partial f_i}{\partial z_j}\bigr]$ has rank $n$ at each point of $\bar G$;
\item for $F:=(f_1,...,f_k)$ the map from $\bar G$ to $\C^{k}$, the set $K:=F(\bar G)$ is polynomially convex.
\end{enumerate}
Then the functions $f_1,...,f_k$ restricted to $\bar G$ generate $A(\bar G)$.

\end{theorem}

\begin{proof} For simplicity, we take $k=n$. Since $\bar G$ is bounded the $f_j$ are bounded there so we may assume $|f_j|\leq 1/2$ on $\bar G$. For $r>0$ let $G_r:=\{z\in \C^n: \hbox{dist}(z,\bar G) < r\}$. By hypothesis, for $\epsilon' >0$ sufficiently small $f_j$ are holomorphic in $G_{\epsilon'}$. From 1., there exists $\epsilon'' >0$ so that $F=(f_1,...,f_n)$ is one-to-one on $G_{\epsilon''}$. Then from 2., we can fix $\epsilon  \leq \min[\epsilon',\epsilon'']$ so that $\bigl[ \frac{\partial f_i}{\partial z_j}\bigr]$ has rank $n$ at each point of $G_{\epsilon}$ and hence for each $z^o\in G_{\epsilon}$, $F$ is a biholomorphic mapping on a neighborhood of $z^o$ in $G_{\epsilon}$. 

Given $a\in G_{\epsilon}\setminus \bar G$, using 3., we can find a polynomial $Q_a$ such that 
$$Q_a(F(a))> \max_{\bar G} |Q_a(F)|.$$
Define $g_a:= Q_a(F)$. Note that $g_a\in [f_1,...,f_n|\bar G]$ and we have $|g_a(a)|> ||g_a||_{\bar G}$. By rescaling, we can assume $|g_a(a)|> 1 >||g_a||_{\bar G}$ and then 
by replacing $g_a$ by a power $g_a^m$ we can assume 
$$|g_a(a)| >2 \ \hbox{and} \ ||g_a||_{\bar G} < 1/2.$$
Thus there exists a neighborhood $N_a$ of $a$ in $G_{\epsilon}$ on which $|g_a|>2$. We do this construction for each $a\in \bar G_{3\epsilon/4}\setminus G_{\epsilon/2}$; the corresponding neighborhoods $N_a$ cover this compact set. Take a finite subcover $N_{a_1},...,N_{a_s}$ and consider the corresponding functions 
$g_{a_1},...,g_{a_s}\in [f_1,...,f_n|\bar G]$. 

Define the mapping $\Phi: G_{3\epsilon/4}\to \C^{n+s}$ via 
$$\Phi(z):=(F(z),g_{a_1}(z),...,g_{a_s}(z))=(f_1(z),...,f_n(z),g_{a_1}(z),...,g_{a_s}(z)).$$
Since $F$ is one-to-one on $G_{3\epsilon/4}$, so is $\Phi$. Let $\Delta$ be the open unit polydisk in $ \C^{n+s}$ and define the non-empty set
$$V:= \Phi(G_{3\epsilon/4}) \cap \Delta.$$
Following, mutatis mutandis, the argument on p.\ 133 of \cite{W}, it follows that $V$ is an analytic subvariety of $\Delta$. (To see this, we first show that $V$ is relatively closed in $\Delta$. Take a sequence $\{p_n\}\subset V$ with $p_n\to p\in \Delta$. We claim $p_n=\Phi(z^{(n)})$ where $z^{(n)}\in G_{\epsilon/2}$. For if $z^{(n)}\in G_{3\epsilon/4}\setminus G_{\epsilon/2}$ then $z^{(n)}\in N_{a_j}$ for some $j\in \{1,...,s\}$. But then $|g_{a_j}(z^{(n)})|> 2$ contradicting $\Phi(z^{(n)})\in \Delta$. Thus $z^{(n)}\in G_{\epsilon/2}$ and hence there exists a limit point $z$ of $\{z^{(n)}\}$ with $z\in \bar G_{\epsilon/2}$. Without loss of generality, $z^{(n)}\to z$ and hence $\Phi(z^{(n)})\to \Phi(z)$; thus $p=\Phi(z)$ and $p\in V$. It remains to show $V$ is locally defined by holomorphic functions. For $p\in V$, we have $p=\Phi(z^o)$ for some $z^o\in  G_{3\epsilon/4}$. Since $F$ is biholomorphic in a neighborhood of $z^o$, each coordinate function $z_{\alpha}$ is a holomorphic function of $f_1,...,f_n$ and hence each component of $\Phi$ is a holomorphic function of $f_1,...,f_n$ in a neighborhood of $z^o$. Thus $V$ is defined locally near $p$ by holomorphic functions.)

Now for each $j=1,...,n$ define the function $Z_j$ on $V$ as 
$$Z_j(p):=z_j(\Phi^{-1}(p)).$$ 
These functions are holomorphic on $V$. A standard extension result allows us to find $H_j$ holomorphic in $\Delta$ with $H_j=Z_j$ on $V$ (e.g., the Proposition on p. 132 \cite{W}). Now 
$\Phi(\bar G)\subset \Phi(G_{3\epsilon/4})$ and $\Phi(\bar G)$ is compact. Since each $|f_k|, |g_{a_l}|\leq 1/2$ on $\bar G$, $\Phi(\bar G)\subset \Delta$ and 
$\Phi(\bar G)$ is compact in $V$. Thus if we fix $j\in \{1,...,n\}$, the Taylor series at the origin of $H_j$ in $\Delta$ converges uniformly on $\Phi(\bar G)$. Taking 
partial sums and returning to $\bar G$, we get a sequence of polynomials in the functions $f_1,...,f_n,g_{a_1},...,g_{a_s}$ which converge uniformly on $\bar G$ to 
$H_j(\Phi)$. But on $V$, $H_j=Z_j$ so  
$$H_j(\Phi)=Z_j(\Phi)=z_j, \ j=1,...,n \ \hbox{on} \ \bar G.$$
As observed earlier, each $g_{a_k} \in [f_1,...,f_n|\bar G]$ so we conclude that for $j=1,...,n$,
$$z_j \in [f_1,...,f_n|\bar G].$$
Thus $P(\bar G) \subset [f_1,...,f_n|\bar G]$. Since $ [f_1,...,f_n|\bar G]\subset A(\bar G)$ and $A(\bar G)=P(\bar G)$, we are done.

\end{proof}

Note that one simply strengthens the first two of the necessary conditions (1) - (3) on page 131 of \cite{W} to require more regularity of $f_1,...,f_k$ at the boundary. Indeed, if one is interested solely in the algebra $P(\bar G)$, for general bounded $G$, assuming $\bar G$ is polynomially convex, the proof above is valid to show $ [f_1,...,f_k|\bar G]=P(\bar G)$ (note by the Oka-Weil theorem, $f$ holomorphic in a neighborhood of $\bar G$ implies $f\in P(\bar G)$). 

\begin{corollary} \label{4.2} Let $G\subset \C^n$ be a bounded domain with the property that $\bar G$ is polynomially convex. Let $f_1,...,f_k$ ($k \geq n$) be holomorphic in a neighborhood of $\bar G$ satisfy the following three properties:
\begin{enumerate}
\item $f_1,...,f_k$ separate points in $\bar G$;
\item the matrix $\bigl[ \frac{\partial f_i}{\partial z_j}\bigr]$ has rank $n$ at each point of $\bar G$;
\item for $F:=(f_1,...,f_k)$, $K:=F(\bar G)$ is polynomially convex (as a subset of $\C^k$).
\end{enumerate}
Then the functions $f_1,...,f_k$ restricted to $\bar G$ generate $P(\bar G)$.

\end{corollary}

\noindent In particular, one could apply this to polynomials $f_1,...,f_k$. 

\begin{remark} Domains $G\subset \C^n, \ n>1$ for which $A(\bar G)=P(\bar G)$ include bounded, strictly pseudoconvex domains with smooth boundary \cite{H}. For such domains, a different proof of Theorem \ref{wermer} was given in \cite{SW}. We add that if $G\subset \C^n$ is bounded and strictly pseudoconvex with $C^2$ boundary and $\bar G$ is polynomially convex, then sufficiently small $C^2$ perturbations of $\bar G$ are also polynomially convex (\cite{Stout}, p. 402). If we begin with such a domain and $f_1,...,f_k$ holomorphic in a neighborhood $U$ of $\bar G$ satisfying condition 3. of Corollary \ref{4.2} and conditions 1. and 2. on the neighborhood $U$, then for sufficiently small $C^2$ perturbations $\bar G'\subset U$ of $\bar G$ we can insure that $\bar G'$ and $K':=F(\bar G')$ are polynomially convex. Hence the same $f_1,...,f_k$ generate $P(\bar G')$. 

Product domains $G=G_1 \times \cdots \times G_n$ where $G_i\subset \C$ is bounded with $\C \setminus \bar G_i$ connected (so that $A(\bar G_i)=P(\bar G_i)$) also satisfy $A(\bar G)=P(\bar G)$, \cite{GN}.

\end{remark}

Thus, working in $\C^n$ with coordinates $(z_1,...,z_n)$, given $W=W(z_1,...,z_n)$ holomorphic on $\bar G$, it seems natural to consider the algebra 
$$[W,z_{j_1}W,...,z_{j_{n-1}}W|\bar G]$$
for $n-1$ of the coordinate functions $z_{j_1},...,z_{j_{n-1}}$. This algebra contains certain linear combinations of weighted polynomials $W^np_n(z_1,...,z_n)$. 
For example, in $\C^2$, consider the algebra generated by $W$ and $z_1W$. Given a polynomial $Q\in \C[z_1,z_2]$, say of degree $m$, so that
$$Q(z_1,z_2)=\sum_{j,k \geq 0, \ j+k \leq m} a_{jk}z_1^j z_2^k,$$
we have
$$Q(W,z_1W)= \sum_{j,k \geq 0, \ j+k \leq m} a_{jk}W^j (z_1W)^k=\sum_{j,k \geq 0, \ j+k \leq m} a_{jk} W^{j+k}z_1^k.$$
However, although $[W,z_{j_1}W,...,z_{j_{n-1}}W|\bar G]\subset A(\bar G)$, unlike the univariate case, we do not necessarily have
\begin{equation} \label{incl} A(\bar G,W)\subset [W,z_{j_1}W,...,z_{j_{n-1}}W|\bar G].\end{equation}
Indeed, if $W=W(z_{j_1},...,z_{j_{n-1}})$ and $p_n$ depends on the missing variable, clearly $W^np_n\not\in [W,z_{j_1}W,...,z_{j_{n-1}}W|\bar G]$. In $\C^2$, for example, 
$W(z_1)^n z_2^n \not \in [W,z_1W|\bar G]$ as $Q(W,z_1W)$ is independent of $z_2$. Of course these examples are quite special; and, in the case where we do have the inclusion (\ref{incl}), the proof of Proposition 2.3 of \cite{PV2} applies mutatis mutandis to show that $A(\bar G,W)$ is a closed ideal of $ [W,z_{j_1}W,...,z_{j_{n-1}}W|\bar G]$ and in this setting 
$$A(\bar G,W)\subset  [W,z_{j_1}W,...,z_{j_{n-1}}W|\bar G] \subset A(\bar G).$$
In the next section, we apply Theorem \ref{wermer} to construct explicit examples where we have $[W,z_{j_1}W,...,z_{j_{n-1}}W|\bar G] = A(\bar G)$.

\section{Examples in $\C^2$} The biggest obstacle to overcome in order to apply Theorem \ref{wermer} to conclude the equality $[W,z_{j_1}W,...,z_{j_{n-1}}W|\bar G] = A(\bar G)$ is 
3.: $K:=F(\bar G)$ is polynomially convex. In this section, we give some explicit examples where this can be verified; for simplicity, we work in $\C^2$. 

We begin with some preliminaries on polynomial convexity. Note that the function $F$ constructed in Theorem \ref{wermer} was a biholomorphic map on a neighborhood of $\bar G$. Unfortunately, the image of a polynomially convex set (such as $\bar G$ in the theorem) under a biholomorphic map need not be polynomially convex; Wermer himself constructed an example which can be found in Example 1 on pp. 135-136 of \cite{W}. For the reader's convenience, we list a few known facts/results on polynomial convexity.

\begin{enumerate}
\item  (\cite{Stout}, Theorem 1.6.24) Let $F:\C^n \to \C^n$ be a proper holomorphic map. For $X\subset \C^n$ compact, we have $\hat X=X$ if and only if $\hat {F^{-1}(X)}=F^{-1}(X)$, and $P(X)=C(X)$ if and only if $P(F^{-1}(X))=C(F^{-1}(X))$.

\item ( \cite{Stout}, Theorem 1.3.11) For $K\subset \C^n$ compact, $\hat K$ coincides with 
$$\hat K_{psh}:=\{z\in \C^n: u(z) \leq sup_K u, \ \hbox{all} \ u \ \hbox{plurisubharmonic in} \ \C^n\}.$$

\end{enumerate}

\noindent (For the Wermer example, he constructs a biholomorphic image of a polynomially convex set (an ellipsoid) in $\C^3$ under the polynomial mapping 
$$F(z_1,z_2,z_3):=(z_1,z_1 z_2 +z_3,z_1z_2^2 -z_2+2z_2 z_3).$$
Here $\det JF=0$ if $z_3=1/2$, and one sees easily that $F(0,z_2,1/2)=(0,1/2,0)$; hence $F:\C^3 \to \C^3$ is not proper.)

Working in $\C^2$, in order to obtain the equality $ [W,z_1W] =A(\bar G)$ for $G$ satisfying $\bar G$ is polynomially convex and 
$A(\bar G)=P(\bar G)$, we seek $W$ on $\bar G$ such that $F=(f_1,f_2)=(W,z_1W)$ satisfy 1.-3. of Theorem \ref{wermer}. A simple example of a holomorphic $W$ is the choice $W(z_1,z_2)=z_2$. Thus we consider $F=(f_1,f_2)=(z_2,z_1z_2)$. It is easy to check that as long as $\bar G\cap \{(z_1,0): z_1\in \C\}=\emptyset$, conditions 1. and 2. of Theorem \ref{wermer} hold. 
Note that $F(z_1,z_2)=(z_2,z_1z_2)$ is not proper as a map from $\C^2$ to $\C^2$. For example, points of the form $(z_1,0)$ with $|z_1| \to \infty$ have as image $F(z_1,0)=(0,0)$ (e.g., the pre-image of $(0,0)$ is not compact). Thus we cannot appeal to the first result above on polynomial convexity to conclude $F(\bar G)$ is polynomially convex.

\begin{example} A bidisk $G$ satisfies $\bar G$ is polynomially convex and 
$A(\bar G)=P(\bar G)$ (cf., \cite{GN} for this second point). Thus we consider
$$G:=\{(z_1,z_2): |z_1-a_1|<r_1, \ |z_2 -a_2|<r_2\} \ \hbox{with} \ r_2 < |a_2|$$
so that $\bar G\cap \{(z_1,0): z_1\in \C\}=\emptyset$. We show that $K:=F(\bar G)$ is polynomially convex. Note
$$K=\{(w_1,w_2)=(z_2,z_1z_2): |z_1-a_1|\leq r_1, \ |z_2 -a_2| \leq r_2\} $$
$$=\{(w_1,w_2): |\frac{w_2}{w_1}-a_1|\leq r_1, \ |w_1-a_2|\leq r_2\}.$$
Fix $(w_1^o,w_2^o)\not \in K$. If $|w_1^o -a_2| > r_2$, taking $Q(w_1,w_2)=w_1-a_2$ shows $(w_1^o,w_2^o)\not \in \hat K$. 
If $|w_1^o -a_2| \leq r_2$ but $|\frac{w_2^o}{w_1^o}-a_1| =(1+\delta) r_1$ with $\delta >0$, we show that by taking $\epsilon =\epsilon(K,\delta, w_2^o)>0$ 
sufficiently small, and choosing a univariate polynomial $p(w_1)$ with 
$$|p(w_1)-1/w_1| < \epsilon \ \hbox{for} \ |w_1-a_2|\leq r_2,$$
 setting $Q(w_1,w_2)=w_2p(w_1)-a_1$ shows $(w_1^o,w_2^o)\not \in \hat K$. To see this, first of all, for $(w_1,w_2)\in K$, since $|w_1|\leq |a_2|+r_2$ and $|w_2|\leq |w_1|(|a_1|+r_1)$, 
 $$|w_2|\leq |w_1|(|a_1|+r_1)\leq  (|a_2|+r_2)(|a_1|+r_1)$$
 and hence
 $$||Q||_K=||w_2p(w_1)-a_1||_K=||w_2[p(w_1)-1/w_1]+w_2/w_1-a_1||_K$$
 $$\leq ||w_2[p(w_1)-1/w_1]||_K+ ||w_2/w_1-a_1||_K\leq  (|a_2|+r_2)(|a_1|+r_1)\epsilon + r_1.$$
 On the other hand,
 $$|Q(w_1^o,w_2^o)|= |w_2^o p(w_1^o)-a_1|= |w_2^o [p(w_1^o)-1/w_1^o]+w_2^o/w_1^o-a_1|$$
 $$\geq |w_2^o/w_1^o-a_1|-  |w_2^o| |p(w_1^o)-1/w_1^o|> (1+\delta)r_1 -|w_2^o|\epsilon.$$
 Thus, given $(w_1^o,w_2^0)$ and hence $\delta$, if we take $\epsilon >0$ satisfying
 $$\epsilon < \frac{\delta r_1}{|w_2^o|+  (|a_2|+r_2)(|a_1|+r_1)},$$
 then $|Q(w_1^o,w_2^o)| > ||Q||_K$.
 
 This shows that $[W,z_1W]=[z_2, z_1 z_2]=A(\bar G)$. What can we say about $A(\bar G,W)$? Let $f(z_1)$ be holomorphic in a neighborhood of 
 the closed disk $\{z_1:|z_1-a_1|\leq r_1\}$ with $f_1 \not \equiv 0$. Recall we are assuming $|a_2| > r_2$; for simplicity, suppose $a_2>0$. If $r_2 > a_2/3$, we claim 
 that $\tilde f(z_1,z_2):= f(z_1)\not \in A(\bar G,W)$. For suppose $\tilde f(z_1,z_2):= f(z_1) \in A(\bar G,W)$ so that there exist a sequence of 
 bivariate polynomials $p_n(z_1,z_2)$ with $W^np_n= z_2^np_n(z_1,z_2)\to \tilde f(z_1,z_2)$ on $\bar G$. In particular, for each 
 fixed $z_1^o$ with $|z_1^o - a_1| \leq r_1$, 
 $$z_2^np_n(z_1^o,z_2)\to f(z_1^o) \ \hbox{on} \ |z_2-a_2|\leq r_2.$$
 But $\{z_2^np_n(z_1^o,z_2)\}$ are weighted holomorphic univariate polynomials, and by \cite{PV1}, if $r_2 > a_2/3$ we arrive at a 
 contradiction provided $f(z_1^o)\not =0$.

\end{example}

\begin{example} We suspect that the Euclidean ball
$$G:=\{(z_1,z_2): |z_1-a_1|^2 + |z_2 -a_2|^2<r_2^2\} \ \hbox{with} \ r_2 < |a_2|$$
should also work with the same mapping $F(z_1,z_2)=(z_2,z_1z_2)$. We were not able to show directly that $K:=F(\bar G)$ is polynomially convex. However, we can consider a 
slightly different domain
$$\tilde G:=\{(z_1,z_2): |z_1z_2-a_1|^2 + |z_2 -a_2|^2<r_2^2\} \ \hbox{with} \ r_2 < |a_2|.$$
A direct calculation of the complex Hessian $H(\rho)$ of the smooth plurisubharmonic defining function
$$\rho(z)= |z_1z_2-a_1|^2 + |z_2 -a_2|^2-r_2^2$$
shows that $\det H(\rho)= |z_2|^2$ so that, since $\bar {\tilde G}\cap \{(z_1,0): z_1\in \C\}=\emptyset$, $\tilde G$ is strictly 
pseudoconvex with smooth boundary and hence satisfies $A(\bar {\tilde G})=P(\bar {\tilde G})$. Moreover, since $\rho$ is continuous and plurisubharmonic on all of $\C^2$ with 
$\bar {\tilde G} =\{(z_1,z_2)\in \C^2: \rho(z_1,z_2)\leq 0\}$, $\bar {\tilde G}$ is polynomially convex (the polynomial hull coincides with the hull with respect to plurisubharmonic functions on $\C^2$). With $F(z_1,z_2)=(z_2,z_1z_2)$, the set $\tilde K :=F(\bar {\tilde G})$ becomes a Euclidean ball:
$$\tilde K=\{(w_1,w_2): |w_1-a_1|^2 + |w_1 -a_2|^2 \leq r_2^2\}.$$
Hence $[W,z_1W]=[z_2, z_1 z_2]=A(\bar {\tilde G})$.

\end{example}

\end{document}